\documentclass[12pt,a4paper]{amsart}
\usepackage{aliascnt}
\usepackage{hyperref}
\usepackage{amsfonts}
\usepackage{mathrsfs}
\usepackage{amsmath}
\usepackage{amsmath,graphics}%diagrams,
\usepackage{amssymb}
\usepackage{indentfirst,latexsym,bm,amsmath,pstricks,amssymb,amsthm,graphicx,fancyhdr,float,color}
\usepackage[all]{xy}
\usepackage{amsthm}
\usepackage{amscd}
\usepackage[all]{hypcap}

\newtheorem{theorem}{Theorem}[section]

\newaliascnt{lemma}{theorem}
\newtheorem{lemma}[lemma]{Lemma}
\aliascntresetthe{lemma}

\newaliascnt{conjecture}{theorem}
\newtheorem{conjecture}[conjecture]{Conjecture}
\aliascntresetthe{conjecture}

\newaliascnt{proposition}{theorem}
\newtheorem{proposition}[proposition]{Proposition}
\aliascntresetthe{proposition}

\newaliascnt{corollary}{theorem}
\newtheorem{corollary}[corollary]{Corollary}
\aliascntresetthe{corollary}

\newaliascnt{problem}{theorem}

\aliascntresetthe{problem}

\newaliascnt{question}{theorem}
\newtheorem{question}[question]{Question}
\aliascntresetthe{question}

\newaliascnt{claim}{theorem}

\aliascntresetthe{claim}

\theoremstyle{definition}

\newaliascnt{definition}{theorem}
\newtheorem{definition}[definition]{Definition}
\aliascntresetthe{definition}

\newaliascnt{example}{theorem}

\aliascntresetthe{example}

\theoremstyle{remark}

\newaliascnt{remark}{theorem}
\newtheorem{remark}[remark]{Remark}
\aliascntresetthe{remark}

\newaliascnt{remarks}{theorem}

\aliascntresetthe{remarks}

\numberwithin{equation}{section}

\numberwithin{figure}{section}

\def\ol{\overline}
\def\ra{\rightarrow}

\def\({$($}
\def\){$)$}

\def\Pic{\text{{\rm Pic\,}}}
\def\rank{\text{{\rm rank\,}}}

\newcommand{\Ecal}{\mathcal{E}}
\newcommand{\Fcal}{\mathcal{F}}
\newcommand{\Mcal}{\mathcal{M}}

\newcommand{\Acal}{\mathcal{A}}

\newcommand{\Hcal}{\mathcal{H}}

\newcommand{\Jac}{\mathrm{Jac}}
\newcommand{\Tcal}{\mathcal{T}}

\newcommand{\Cbb}{\mathbb{C}}
\newcommand{\Fbb}{\mathbb{F}}

\newcommand{\Qbb}{\mathbb{Q}}
\newcommand{\Rbb}{\mathbb{R}}

\newcommand{\Gbf}{\mathbf{G}}
\newcommand{\Hbf}{\mathbf{H}}

\newcommand{\Ubf}{\mathbf{U}}

\newcommand{\Bbar}{{\ol B}}

\newcommand{\Sbar}{{\ol S}}
\newcommand{\fbar}{{\ol f}}

\newcommand{\pfrak}{\mathfrak{p}}

\newcommand{\Res}{\mathrm{Res}}
\newcommand{\bsh}{\backslash}
\newcommand{\ad}{\mathrm{ad}}
\newcommand{\der}{\mathrm{der}}
\newcommand{\isom}{\simeq}
\newcommand{\mono}{\hookrightarrow}

\newcommand{\Sp}{\mathrm{Sp}}

\newcommand{\GL}{\mathrm{GL}}

\newcommand{\SU}{\mathrm{SU}}

\newcommand{\inv}{{-1}}

\newcommand{\half}{\frac{1}{2}}

\newcommand{\Ccal}{\mathcal{C}}
\newcommand{\kbar}{{\bar{k}}}

\newcommand{\Jcal}{\mathcal{J}}

\newcommand{\Sing}{{\mathrm{Sing}}}
\newcommand{\Spec}{{\mathrm{Spec}}}

\newcommand{\Nbb}{\mathbb{N}}

\newcommand{\xbar}{\bar{x}}

\newcommand{\Qac}{\bar{\Qbb}}
\newcommand{\Fal}{\mathrm{Fal}}

\newcommand{\MT}{\mathrm{MT}}
\newcommand{\Disc}{\mathrm{Disc}}
\newcommand{\wk}{\mathrm{wk}}

\title{On a question of Ekedahl and Serre}

\author{Ke CHEN}

\address{Department of mathematics, Nanjing University, Nanjing, China, 210093}
\email{kechen@nju.edu.cn}

\author{Xin Lu}
\address{Institut f\"ur Mathematik, Universit\"at Mainz, Mainz, Germany, 55099}
\email{x.lu@uni-mainz.de}
\author{Kang Zuo}
\address{Institut f\"ur Mathematik, Universit\"at Mainz, Mainz, Germany, 55099}
\email{zuok@uni-mainz.de}

\thanks{This work is supported by SFB/Transregio 45 Periods, Moduli Spaces and Arithmetic of Algebraic Varieties of the DFG (Deutsche Forschungsgemeinschaft),
	and partially supported by National Key Basic Research Program of China (Grant No. 2013CB834202) and National Natural Science Foundation of China, Grant No. 11301495, and Fundamental Research Funds for the Central Universities,  Nanjing University, no. 0203-14380009}


\begin{document}

\maketitle

\begin{abstract}
	In this paper we study various aspects of the Ekedahl-Serre problem. We formulate questions of Ekedahl-Serre type and Coleman-Oort type for general weakly special subvarieties in the Siegel moduli space, propose a conjecture relating these two questions, and provide examples supporting these questions. The main new result is an upper bound of genera for curves over number fields whose Jacobians are isogeneous to products of elliptic curves satisfying the Sato-Tate equidistribution, and we also refine previous results showing that certain weakly special subvarieties only meet the open Torelli locus in at most finitely many points.%The main new result is the existence of an upper bound of genera for curves over number fields whose Jacobians are isogenous to self-products of elliptic curves satisfying the Sato-Tate equidistribution, and we also refine a previous result showing that certain weakly special subvarieties of unitary type only meet the open Torelli locus in at most finitely many points.
	% the Ekedahl-Serre conjecture over number fields. The main result is the existence of an upper bound for the genus of curves whose Jacobians admit isogenies of bounded degrees to self-products of a given elliptic curve over a number field satisfying the Sato-Tate equidistribution, and the technique is motivated by similar results over function field due to Kukulies. A few variants are considered and questions involving more general Shimura subvarieties are discussed.
\end{abstract}

\tableofcontents

\section{Introduction}\label{sec-introduction}

This paper is dedicated to a question of Ekedahl and Serre \cite{es-93} on  algebraic curves with completely decomposable Jacobians when the genus tends to  infinity. Here a Jacobian, or more generally an abelian variety, is said to be completely decomposable (also called totally split in the literature)
if it is isogenous to a product of elliptic curves.
In \cite{es-93} Ekedahl and Serre have constructed various examples of algebraic curves
with completely decomposable Jacobians, and they asked whether the genera of the algebraic curves
with completely decomposable Jacobians are bounded from above.

The question can be reformulated in terms of Shimura varieties. One considers $M$ the Shimura subvariety in $\Acal_g$ parametrizing abelian varieties isomorphic to a product of elliptic curves respecting the principal polarizations. It is clear that $M$ is isomorphic to a $g$-fold product of modular curves upon the choice of a suitable level structure imposed on $\Acal_g$. An abelian variety isogenous to a product of elliptic curves thus corresponds to a point in some $T_q(M)$, where $T_q$ stands for the Hecke translation by  $q\in\Sp_{2g}(\Qbb)$ in $\Acal_g$. The question of Ekedahl and Serre is concerned with the existence of a lower bound for $g$ such that the intersection $$\Tcal_g^\circ\bigcap\bigg(\bigcup_{q\in\Sp_{2g}(\Qbb)}T_q(M)\bigg)$$ is finite or even empty, where $\Tcal_g^\circ$ is the open Torelli locus parameterizing Jacobians of smooth projective curves of genus $g$.

The huge union $\bigcup_{q\in\Sp_{2g}(\Qbb)}T_q(M)$ appears inconvenient to deal with, and the situation would be considerably simplified if one could reduce the intersection above to a ``smaller'' one like $$\Tcal_g^\circ\bigcap\big(T_{q_1}(M)\cup\cdots\cup T_{q_N}(M)\big)$$ involving only finitely many Hecke translates. This would transform the question of Ekedahl and Serre to a problem in the flavor of Coleman and Oort: the original Coleman-Oort conjecture, cf.\cite{mo-13}, predicts that the when $g$ is large enough, the intersection of $\Tcal_g^\circ$ with any Shimura subvariety in $\Acal_g$ of strictly positive dimension should contain at most finitely many CM points, and here we ask for condition under which the intersection $\Tcal_g^\circ\bigcap(T_{q_1}(M)\cup\cdots\cup T_{q_N}(M))$, with non-CM points allowed, is finite or even empty.

We expect that these questions still make good sense for more general Shimura subvarieties and even weakly special subvarieties in $\Acal_g$. In this paper we present evidence supporting the perspective. The materials are organized as follows:

\autoref{sec-questions} formulates the questions of Ekedahl-Serre type and of Coleman-Oort type for weakly special subvarieties in $\Acal_g$, and we propose a conjecture relating these two questions. The notion of weakly special subvarieties is recalled and the relation between our questions and the Zilber-Pink conjecture is briefly explained.

\autoref{sec-unitary} provides evidence toward the questions for weakly special subvarieties in strictly positive dimension. We extend our previous results on the Coleman-Oort conjecture to show that certain weakly special subvarieties only meet $\Tcal_g^\circ$ in at most finitely many points: \begin{theorem}
	Let $Z\subset\Acal_g$ be a weakly special subvariety with weak Mumford-Tate group $\Gbf$ such that $\Gbf\otimes_\Qbb\Rbb\isom\prod_{i=1,\cdots,r}\SU(p_i,q_i)$ with $p_i+q_i=n$  independent of $i=1,\cdots,r$, satisfying the following conditions: \begin{itemize}
		
		\item $p_i\geq q_i\geq 0$ and $\max_{i=1,\cdots,r}\{\frac{q_i}{n}\}<\frac{1}{12}$;
		
		\item $q_{i_0}\geq 2$ for some $i_0\in\{1,\cdots,r\}$. 
		
		\end{itemize} Then $Z$ only meets $\Tcal_g^\circ$ in at most finitely many points.
\end{theorem} Here the notion of weak Mumford-Tate group is defined as an analogue of the derived part of the Mumford-Tate group of a Shimura subvariety, and  the main technique in the proof is a slope inequality for the Hodge bundles associated to surface fibrations.

\autoref{sec-ekedahl-serre} provides evidence
supporting the questions for certain  type of points, i.e., zero-dimensional weakly special subvarieties. We first treat the case of self-products of an elliptic curve:
% The main results are the following:

%\begin{theorem}\label{thm bound} Fix real constants $c>0, H>0$ and an integer $d>0$.  Then there exists a constant $G=G(c,d,H)$ such that the following is true: let $E$ be an elliptic curve over a number field $F$ of degree $d$ and Faltings height not exceeding $H$ which satisfies the Sato-Tate equidistribution, then for any smooth projective curve $C$ over $F$ of genus $g$ admitting an isogeny from $J=\Jac(C)$ to $E^g$ of degree at most $c$ must hold $g\leq G$.
	
%\end{theorem}

\begin{theorem}\label{thm-ekedahl-serre}
	Let $c$, $d$, and $H$ be given positive constants.
	Then there exists a positive constant $G=G(c,d,H)$, such that $g\leq G$
	if there exists a smooth projective curves $C$ over $\Qac$ of genus $g$ satisfying the following conditions:
	\begin{itemize}
		\item $C$ is defined over a number field $F$ of degree at most $d$;		
		
		\item the Jacobian $\Jac(C)$ of $C$ admits an isogeny of degree at most $c$ to $E^g$,
		where $E$ is an elliptic curve over $F$ with Faltings height
		$h_{\Fal}(E)$ not exceeding $H$ and satisfying the Sato-Tate equidistribution
		 as in Theorems \ref{equid-ncm} and \ref{equid-cm}.
	\end{itemize}	
\end{theorem}

Restricting to the CM case, it is known that the Sato-Tate equidistribution holds
for a finite product of CM elliptic curves (cf. \autoref{johansson} for general CM abelian varieties, not necessarily simple),
and that the Faltings height of a CM elliptic curve is bounded from above by the degree of its definition field (cf. \autoref{thm-bound-falting}).
Thus we prove
\begin{theorem}\label{thm-cm-elliptic}
	Let $c$ and $d$ be given positive constants. Consider the following set of isomorphism classes:
		$$\mathfrak{C}_d^c=\left\{~C~\left|~
		\begin{aligned}
		& C\text{~is a smooth projective curve of genus $g>0$ defined over}\\
		&\text{a number field $F$ with $[F:\mathbb{Q}]\leq d$, and the Jacobian $\Jac(C)$}\\
		&\text{admits an isogeny of degree at most $c$ to $E_1\times\cdots\times E_g$}\\
		&\text{for some CM elliptic curves $E_i$ over $F$.}
		\end{aligned} \right.\right\}\Bigg/_{\hspace{-2mm}\Large\cong}~.$$	
	Then the set $\mathfrak{C}_d^c$ is finite.
	In particular, the genera of curves in $\mathfrak{C}_d^c$ are bounded.
\end{theorem}

\begin{theorem}\label{thm-ekedahl-serre-2}
	Let $d$ be any given positive constant, and consider the following set of isomorphism classes
	$$\mathfrak{S}_d=\left\{~C~\left|~
	\begin{aligned}
	& C\text{~is a smooth projective curve of genus $g>0$ defined over}\\
	&\text{a number field $F$ with $[F:\mathbb{Q}]\leq d$, and the Jacobian $\Jac(C)$}\\
	&\text{is isogenous to $E^g$ for some CM elliptic curve $E$ over $F$.}
	\end{aligned} \right.\right\}\Bigg/_{\hspace{-2mm}\Large\cong}~.$$
	Then the set $\mathfrak{S}_d$ is finite.
	In particular, the genera of curves in $\mathfrak{S}_d$ are bounded. 
\end{theorem}

%Restricting to the CM case, we can get the following finiteness result.

Similar arguments allow us to treat the case of products of CM abelian varieties whose dimensions and number fields of definition are both bounded \`a priori, see \autoref{thm-4.14}.
In other words, we exclude from $\Tcal_g^\circ$ certain types of algebraic points over number fields,
as long as the abelian varieties represented by these points satisfy the Sato-Tate equidistribution.
The proof is much motivated by a similar result in positive characteristic by \cite{kuk-10}
based on a careful study of singular points in surface fibration.

%In other words we exclude from $\Tcal_g^\circ$ certain types of CM points as well as some algebraic points over number fields, the Mumford-Tate groups of which are either  $\Res_{K/\Qbb}\Gbb_\mrm$ or $\GL_2$ ``diagonally'' embedded in $\GSp_{2g}$ up to Hecke translation of bounded degree; the derived Mumford-Tate groups are trivial in these cases. Here $K$ is some imaginary quadratic number field that arises in the CM case. The proof relies on the known results of Sato-Tate equidistribution over number fields, and is much motivated by a similar treatment in positive characteristic by \cite{kuk-10} using Sato-Tate equidistribution over function fields. Note that a bound on degrees of isogeny is imposed, and this theorem only treats the simplified version of the Ekedahl-Serre question instead of the original one involving the whole Hecke orbit. Although emphasis is put on Hecke orbits, a natural corollary is presented concerning elliptic curves of bounded Faltings heights over number fields of bounded degrees satisfying Sato-Tate equidistribution, since there are only finitely many of them. %of the bound on genera involving elliptic curves of bounded Faltings heights over a number field of bounded degree is also mentioned. %A few naive variants are also included concerning higher dimensional CM abelian varieties.

%with two variants of the main theorem \ref{thm-main} for higher dimensional CM abelian varieties using recent generalizations of the Sato-Tate conjecture.
\begin{remark}
It should be pointed out that the general situation of Ekedahl-Serre question is still open. The recent works of Paulhus etc. have updated the list of values of $g$ for which there exists totally decomposable Jacobians of genus $g$, cf. \cite{paulhus acta, paulhus open book, paulhus preprint}. %However these aspects are not touched upon in this paper.%Besides, the bound in the main theorem above is not effective, needless to mention the constraint on isogeny degrees.
\end{remark}

\section{Questions of Ekedahl-Serre and Coleman-Oort types}\label{sec-questions}

We refer to \cite{clz-compositio} for our convention on  (connected) Shimura subvarieties in $\Acal_g$, which are associated to (connected) Shimura subdata and are also called special subvarieties in the literature. 

Aside from Shimura subvarieties there is the notion of weakly special subvarieties, the definition of which is briefly recalled as follows. Let $M=\Gamma\bsh X$ be the Shimura variety defined by $(\Gbf,X;X^+)$ and $\Gamma$, and $M'\subset M$  the Shimura subvariety defined by some Shimura subdatum $(\Gbf',X';X'^+)$. There exists a morphism of connected Shimura varieties $p:M'\ra M''$ where $M''$ is the Shimura variety associated $(\Gbf'^\ad,X'^\ad;X'^+)$ using  the adjoint group $\Gbf'^\ad$ of $\Gbf'$, the same connected Hermitian symmetric domain $X'^+$, the orbit $X'^\ad=\Gbf'^\ad(\Rbb)X'^+$, and $p$ is induced by the projection $\Gbf'\ra\Gbf'^\ad$. Once $(\Gbf'^\ad,X'^\ad;X'^+)$ admits a decomposition into a product of Shimura data $(\Gbf_1,X_1;X_1^+)\times(\Gbf_2,X_2;X_2^+)$, one may find a finite morphism of Shimura varieties $M''\ra M_1\times M_2$ with $M_i$ associated to $(\Gbf_i,X_i;X_i^+)$, and an irreducible component of the preimage of $\xbar_1\times M_2$ in $M'$,  for some complex point $\xbar_1\in M_1$, is a weakly special subvariety of $M$, and every weakly special subvariety in $M$ is obtained in this way. In this paper we do allow $M_2$ to be trivial, and in this way the notion of weakly special subvarieties of dimension zero is not very interesting: it is simply the whole collection of complex points in $M$.

It is proved in \cite{moonen linear} that weakly special subvarieties in Shimura varieties are the same as totally geodesic subvarieties, and that a weakly special subvariety is special, i.e. is a Shimura subvariety if and only if it contains a special point, namely zero-dimensional Shimura subvariety. %In general weakly special subvarieties are not necessarily defined over number fields, and this notion is not very interesting in dimension zero as any (complex) point is weakly special.

Similar to the notion of Mumford-Tate groups for Shimura subdata and Shimura subvarieties, we can talk about the weak Mumford-Tate group associated to a weakly special subvariety:

\begin{definition}[weak Mumford-Tate group]\label{derived mumford-tate group} Let $M=\Gamma\bsh X^+$ be a Shimura variety defined by $(\Gbf,X;X^+)$, and let $Z$ be a weakly special subvariety in $M$ defined by the procedure above: one starts with a Shimura subvariety $M'=\wp_\Gamma(X'^+)$ given by $(\Gbf',X';X'^+)$, and $p:M'\ra M''$ the morphism of Shimura varieties given by the adjoint map $$(\Gbf',X';X'^+)\ra(\Gbf'^\ad,X'^\ad;X'^+)\isom(\Gbf_1,X_1;X_1^+)\times(\Gbf_2,X_2;X_2^+),$$
	so that $Z$ is an irreducible component in $p^\inv(\xbar_1\times M_2)$. In this case we define the weak Mumford-Tate group $\MT^\wk(Z)$ to be the derived part of the neutral component of $p^\inv(\Gbf_2)$. From the factorization $p:\Gbf'\ra\Gbf'^\ad\isom\Gbf_1\times\Gbf_2$, it is easy to verify that $\MT^\der(Z)$ is a connected semi-simple $\Qbb$-subgroup of $\Gbf$, with adjoint group $\Gbf_2$.%, and only depends on $Z$ itself, independent of the choice of $M'$ and $p^\inv(\xbar_1\times M_2)$ presenting $Z$ in the manner above.
	\end{definition}
The construction of $\MT^\wk(Z)$ is not canonical.  For example, one may consider a second Shimura subdatum $(\gamma\Gbf'\gamma^\inv,\gamma X;\gamma X^+)$ by the conjugation of some $\gamma\in\Gamma$, which leaves $M'$ and $Z$ unchanged but modifies $\MT^\wk(Z)$. This notion is only introduced as a choice of semi-simple $\Qbb$-subgroup in $\Sp_{2g}$ that gives rise to an equivariant embedding $X_Z^+\mono \Hcal-g^+$ in the sense of \cite{satake rational}, with $X_Z^+$ the Hermitian symmetric  uniformizing $Z$.  In particular,  $Z$ is also of the form $\Gamma_Z\bsh X_Z^+$, where $\Gamma_Z$ is a congruence subgroup in $\MT^\der(Z)$. These observations are useful later for the study of Hodge vector bundles on $Z$ in Section 3.

Note that for $Z=M'$ a Shimura subvariety defined by $(\Gbf',X';X'^+)$ one may simply choose $\MT^\wk(Z)=\Gbf'^\der$. However for a weakly special subvariety, like a point $z$ of zero dimension, its weak Mumford-Tate group is necessarily a point,  namely the trivial group, while the derived part of its Mumford-Tate group might be non-trivial, related to the minimal Shimura subvariety containing it.
%We apologize for the possibly misleading notation $\MT^\der(Z)$: if $z$ is a Hodge generic point in $M$, namely a point not contained in any proper Shimura subvariety $M'\subset M$, then the Mumford-Tate group $\MT(z)$ of $z$ is $\Gbf$, and the derived Mumford-Tate group $\MT^\der(z)$, when viewing $z$ as a zero-dimensional weakly special subvariety, is simply the trivial $\Qbb$-group. In what follows the delicate difference between $\MT(Z)^\der$ and $\MT^\der(Z)$ is not involved, and we prefer to use $\MT^\der(Z)$ especially in Section 3 to emphasize the geometry of the weakly special subvariety $Z$ dominated by the $\Qbb$-group of symmetry. %: the notion of Mumford-Tate group is well-defined for any subvariety $Z\subset M$ as a connected reductive subgroup $\MT(Z)$ inside $\Gbf$ the $\Qbb$-group in the definition of the Shimura variety $M$. 

\begin{definition}

Fix $M=\Gamma\bsh X^+$ a Shimura variety defined by $(\Gbf,X;X^+)$ and $M'$ a Shimura subvariety defined by $(\Gbf',X';X'^+)$. For $q\in\Gbf^\der(\Qbb)$ we define the Hecke translate $T_q(M')$ of $M'$ by $q$ to be the Shimura subvariety defined by $(q\Gbf'q^\inv,qX';qX'^+)$. Namely we realize $M'$ as $\wp_\Gamma(X'^+)$ using the uniformization map $\wp_\Gamma:X^+\ra M$, and we define $T_q(M')=\wp_\Gamma(qX'^+)$. The union $$H_M(M'):=\bigcup_{q\in\Gbf^\der(\Qbb)^+}T_q(M')$$ is called the total Hecke orbit of $M'$ in $M$. 

The same procedure applies to weakly special subvarieties, because it is easy to see that Hecke translation respects weakly special subvarieties. Moreover, if $Z$ is weakly special with associated weak Mumford-Tate group $\Hbf$, then $T_q(Z)$ is weakly special with associated derived Mumford-Tate group $q\Hbf q^\inv$. 

In the sequel we only need the  total Hecke orbit for $M=\Acal_g$ using Hecke translates indexed by $q\in\Sp_{2g}(\Qbb)$, and we simply write $H(M')$ omitting the subscript $\Acal_g$.
\end{definition}

%Let $\Acal_g$ be the Siegel moduli space parametrizing principally polarized abelian varieties with suitable level structures, which is a Shimura variety associated to the datum $(\GSp_{2g},\Hcal_g^\pm;\Hcal_g^+)$. For a Shimura subvariety $M\subset\Acal_g$, we have the total Hecke orbit of $M$ in $\Acal_g$: $$H(M)=\bigcup_{q\in\Sp_{2g}(\Qbb)}T_q(M)$$ where $T_q(M)$ is the Hecke translate of $M$ by $q$: if $M$ is defined by the Shimura subdatum $(\Gbf,X;X^+)$ then $T_q(M)$ is the Shimura subvariety defined by $(q\Gbf q^\inv,qX;qX^+)$. 

The following two questions are inspired by the question of Ekedahl and Serre and the conjecture of Coleman and Oort:
\begin{question}\label{question-clz}
	Fix a weakly special $M$ in $\Acal_g$.
	
	\begin{itemize}
		\item[(ES)]
		Under what condition for $M$ and $g$ could one find $\Tcal_g^\circ \bigcap H(M)$ being finite or even empty?
		
		\item[(CO)] What kind of weakly special $M\subset\Acal_g$ could produce a finite intersection $\Tcal_g^\circ\bigcap M$? When could the intersection  $\Tcal_g^\circ\bigcap M$ be even empty?
	\end{itemize}
\end{question}

Question (CO) is of different flavor from (ES), as only one fixed weakly special variety is considered instead of the total Hecke orbit, and (CO) slightly strengthens the original Coleman-Oort conjecture which only predicts finiteness of CM points in $\Tcal_g^\circ$ for $g$ large.

As we have mentioned in \autoref{sec-introduction},
it is in general inconvenient to deal with the total Hecke orbits directly. Motivated by our progress in (CO),
which will be reviewed later in \autoref{sec-unitary},
we propose the following conjecture which should serve as a bridge from (ES) to (CO):

\begin{conjecture}\label{bridge}
	For $g$ large enough and $M$ a weakly special subvariety in $\Acal_g$, the intersection $\Tcal_g^\circ\bigcap H(M)$ can be reduced to only finitely many Hecke translates of $M$, namely there exists $q_1,\cdots,q_N\in\Sp_{2g}(\Qbb)$ such that $$\Tcal_g^\circ\bigcap H(M)=\Tcal_g^\circ\bigcap\Big(T_{q_1}(M)\cup\cdots\cup T_{q_N}(M)\Big).$$
\end{conjecture} 

Note that Questions (ES) and (CO) and the conjecture above  make sense for zero-dimensional weakly special subvarieties, namely points. If the point in question is a CM point, then so it is with any of its Hecke translates, and (ES) in this case is a consequence of the original Coleman-Oort conjecture which predicts the finiteness of CM points in $\Tcal_g^\circ$ for $g$ large. Tsimerman has shown in \cite{tsimerman torelli} that (ES) holds for certain points which actually lead to empty intersection. Chai and Oort went further in \cite{chai oort jacobian} showing that similar empty intersection also occurs for the Hecke orbit of certain CM points with Hodge generic subvarieties in $\Acal_g$, i.e. closed irreducible subvarieties not contained in any proper Shimura subvarieties. We will provide some zero-dimensional examples supporting (CO) in Section 4.

%The original Coleman-Oort conjecture predicts that there are only finitely many CM points lying in $\calt^\circ_g$ when $g$ is large; in particular, the intersection $\calt^{\circ}_g\cap H(M)$ is contained in finitely many Hecke translates of $M$ for a zero-dimensional Shimura subvariety $M$. At this level, Chai-Oort \cite{chai oort jacobian} proved that certain zero-dimensional Shimura subvarieties, called Weyl CM points loc. cit., indeed satisfy this property. Examples supporting (CO) will be also discussed later in Sections \ref{sec-ekedahl-serre} and \ref{sec-variants}, using  the equidistribution of Sato-Tate type.

The recent progress on the Zilber-Pink conjecture, which is a far-reaching generalization of the Andr\'e-Oort conjecture, is much motivated by the rich arithmetic and geometry of atypical intersection in Shimura varieties: \begin{conjecture}[Pink's conjecture on atypical intersection]\label{zilber pink} For $M$ a Shimura variety, write $M^{[d]}$ for the union of Shimura subvarieties of codimension at least $d$. Then for any Hodge generic closed irreducible subvariety $V\subsetneq M$, the intersection $V\cap M^{[1+\dim V]}$ is contained a finite union $M_1\cup\cdots\cup M_n$ for some Shimura subvarieties $M_1,\cdots,M_n\subsetneq M$.
\end{conjecture}

This conjecture of atypical intersection is a consequence of the more general conjecture of Zilber and Pink, which  should follow from the hyperbolic Ax-Schanuel conjecture and the conjecture of large Galois orbits, see \cite{daw ren} for more details. Note that a proof of the hyperbolic Ax-Schanuel conjecture is recently announced by Mok, Pila, and Tsimerman.
%{\red (need to be checked).}

However it is not yet clear how to draw useful consequences from the Zilber-Pink conjecture towards the conjecture bridging (ES) and (CO). Consider $V=\Tcal_g$ which is Hodge generic in $\Acal_g$, then for $M\subset \Acal_g$ a Shimura subvariety of codimension at least $3g-2$ the total Hecke orbit $H(M)$ is contained in $\Acal_g^{[3g-2]}$, and one should have $\Tcal_g^\circ\cap H(M)\subset V\cap\Acal_g^{[3g-2]}$ be contained in finitely many proper Shimura subvarieties $M_1,\cdots,M_n$ inside $\Acal_g$, but it is not evident that one could choose these Shimura subvarieties to be contained in Hecke translates of $M$.

\section{Evidence in positive dimensions}\label{sec-unitary}

Our previous work \cite{clz-asian} has shown that Question (CO) is true for Shimura varieties whose Mumford-Tate groups contain ``large'' compact factors, which we review in the weakly special setting as follows for reader's convenience:

\begin{proposition}[{\cite[2.4, 2.5]{clz-asian}}]\label{weakly special compact factors} Let $Z\subset\Acal_g$ be a weakly special subvariety of dimension $>0$ with weak Mumford-Tate group $\Gbf=\MT^\der(Z)$. Assume  that $\Gbf=\Res_{F/\Qbb}\Hbf$ for some semi-simple $F$-group with $F$ a totally real number field of degree $d$ such that $\Hbf$ gives rise to a compact Lie group along $r$ real embeddings $F\mono\Rbb$, and gives non-compact Lie groups along the other $d-r$ real embeddings. Then the open Torelli locus $\Tcal_g^\circ$ only meets $M$ in at most finitely many $\Qac$-points when $$\frac{r}{d}>\frac{5}{6}+\frac{1}{6g}$$
	
\end{proposition}

%\begin{proposition}[{\cite[Proposition 2.4 and Corollary 2.5]{clz-asian}}]
	%Let $M\subset\Acal_g$ be a Shimura subvariety defined by a connected Shimura datum $(\Gbf,X)$, such that $\Gbf^\der=\Res_{F/\Qbb}\Hbf$ for some semi-simple $F$-group with $F$ a totally real number field of degree $d$. Assume that $\Hbf$ gives rise to a compact Lie group along $r$ real embeddings $F\mono\Rbb$, and gives non-compact Lie groups along the other $d-r$ real embeddings. Then the open Torelli locus $\Tcal_g^\circ$ only meets $M$ in at most finitely many $\Qac$-points when $$\frac{r}{d}>\frac{5}{6}+\frac{1}{6g}$$
%\end{proposition}
The main idea behind this criterion is a slope inequality of Xiao which allows us to exclude weakly special subvarieties with sufficiently many compact factors in the derived Mumford-Tate groups, using numerical properties of the semi-stable surface fibration associated to a curve in $\Tcal_g^\circ$. In particular the property of possessing ``large'' compact factors is invariant under Hecke translation for a given weak Mumford-Tate group, which makes (ES) more hopeful via the reduction to (CO) through \autoref{bridge}.%For reader's convenience we sketch the proof: %from \cite{clz-compositio} and \cite{clz-asian}.
\begin{proof}[Sketch of proof]
	%The open Torelli locus can be also defined over $\Qac$, it suffices to show that the intersection $\Tcal_g^\circ\cap M$ is zero-dimensional.
	
 If the intersection were of strictly positive dimension, it would contain a curve $C$, which lifts to a curve $B$ inside $\Mcal_g$ and the construction in \cite{clz-compositio} and \cite{clz-asian} completes it into a semi-stable surface fibration $\fbar:\Sbar\ra\Bbar$, whose Hodge bundle $\omega=\fbar_*\omega_{\Sbar/\Bbar}$, a vector bundle on $\Bbar$ of rank $g$, is determined by the (1,0)-part in the Hodge decomposition for $C$ using the universal family of abelian varieties over $C$ given by the modular interpretation $C\subset Z\subset\Acal_g$.
	
	The Hodge bundle $\omega$ on $\Bbar$ admits a decomposition into $\Fcal_0\oplus\Fcal_1$, with $\Fcal_0$ the flat part of degree zero, and the slope inequality of Xiao (cf. \cite[Theorem\,1.2.2]{clz-compositio}) for $\fbar:\Sbar\ra\Bbar$ implies that $$\frac{\rank\Fcal_0}{g}\leq\frac{5}{6}+\frac{1}{6g}.$$  The direct sum $\Fcal_0\oplus\Fcal_1$ is induced from a similar decomposition $\Ecal_0\oplus\Ecal_1$ on $C$ by subbundles of the same ranks respectively, using the (1,0)-part of the Hodge decomposition for the universal family of abelian varieties by the modular interpretation of $C\mono\Acal_g$. The refinement $C\subset Z\subset\Acal_g$ implies that $\frac{\rank\Fcal_0}{g}\geq\frac{\rank\Ecal_0}{g}\geq\frac{r}{d}$, and a contradiction arises when \[\frac{r}{d}>\frac{5}{6}+\frac{1}{6g}.\qedhere\]
\end{proof}

The proposition above only requires information from the ``portion'' of compact factors. The paper \cite{clz-asian} has passed on to Shimura varieties of $\SU(n,1)$-type, i.e. uniformized by the Hermitian symmetric space associated to the simple Lie group $\SU(n,1)$, with similar results true for weakly special subvarieties of $\SU(n,1)$-type. Its argument through the slope filtration of the Higgs bundles could be adapted to treat more general weakly special subvarieties of unitary type:

\begin{definition}\label{shimura varieties of unitary type}
	A weakly special subvariety $Z$ is said to be of unitary type if its weak Mumford-Tate group $\Gbf$ is a simple $\Qbb$-group admitting a decomposition in Lie groups of the following form $$\Gbf\otimes_\Qbb\Rbb\isom\SU(p_1,q_1)\times\cdots\times\SU(p_r,q_r)$$ with $p_i\geq q_i\geq0$ and $p_i+q_i=n$ constant. The universal cover $X_Z^+$ of $Z$ is thus the direct product of the Hermitian symmetric domains $X_i^+$ associated to those $\SU(p_i,q_i)$ with $q_i>0$, and it is equivariantly embedded in $\Hcal_g^+$ with respect to $\Gbf\mono\Sp_{2g}$.
	
	Shimura varieties of unitary type are simply weakly special subvarieties of unitary type containing CM points.
	%A Shimura subvariety of unitary type in $\Acal_g$ is the one associated to a Shimura subdatum in $(\GSp_{2g},\Hcal_g^\pm;\Hcal_g^+)$ of unitary type in the sense above.	
	%(One may also consider $(\Gbf,X;X^+)$ such that $\Gbf^\ad_\Rbb$ admits a decomposition in adjoint groups associated to such unitary groups, but this case is not needed in the sequel and is thus omitted.)
\end{definition}

\begin{theorem}\label{unitary refinement}
%	Let $M\subset\Acal_g$ be a weakly subvariety of unitary type associated to some Shimura subdatum $(\Gbf,X;X^+)$ of $(\GSp_{2g},\Hcal_g^\pm;\Hcal_g^+)$, with $\Gbf^\der\otimes_\Qbb\Rbb\isom\prod_{i=1}^r\SU(p_i,q_i)$, $p_i+q_i=n$ 

Let $Z\subset\Acal_g$ be a weakly special subvariety of unitary type, with weak Mumford-Tate group $\Gbf$ admitting a decomposition $\Gbf\otimes_\Qbb\Rbb\isom\prod_{i=1}^r\SU(p_i,q_i)$, $p_i+q_i=n$, such that: \begin{itemize}
		\item $p_i\geq q_i\geq0$ and $\max\limits_{i=1,\cdots,r}\big\{\frac{q_i}{n}\big\}<\frac{1}{12}$,
		\item $q_{i_0}\geq 2$ for some $i_0\in\{1,\cdots,r\}$.
	\end{itemize}
	Then $Z$ only meets $\Tcal_g^\circ$ in at most finitely many points.
\end{theorem}

The main idea of the proof is similar to \cite[Section\,3]{clz-asian}:

\begin{proof}
	Using the Satake classification cf.\cite{satake rational}, the condition $q_i\geq 2$ for some $i$ forces that the symplectic representation $\Gbf\mono\Sp_{2g}$ of the simple $\Qbb$-group $\Gbf^\der$ decomposes, after the base change from $\Qbb$ to $\Rbb$, into the following form: $$\Rbb^{2g}=V_0\bigoplus(V_1\oplus\cdots\oplus V_r)^{\oplus m}$$ with $V_0$ a trivial representation of even dimension $2n_0$ and $V_i$ the $2n$-dimensional standard $\Rbb$-linear representation of $\SU(p_i,q_i)$ on $\Cbb^{n}=\Cbb^{p_i}\oplus\Cbb^{q_i}$ preserving an Hermitian form of signature $(p_i,q_i)$. In particular, $g=n_0+rmn$.
	
	Fine information is needed for the Hodge decomposition of the complex vector bundle associated to the locally constant sheaf associated to the $\Cbb$-linearized representation $\Gbf^\der(\Rbb)\ra\GL_{2g}(\Cbb)$. Write $\Ecal=\Ecal^{1,0}\oplus\Ecal^{0,1}$ for the Hodge decomposition on $Z$ given by the modular interpretation of $Z\mono\Acal_g$, with $\Ecal^{1,0}$ and $\Ecal^{0,1}$ both complex vector bundles of rank $g$:
	\begin{itemize}
		\item $\Gbf^\der(\Rbb)$ acts on $V_0$ trivially, and $V_0$ contributes to both $\Ecal^{1,0}$ and $\Ecal^{0,1}$ a trivial vector bundle of rank $\dim V_0$.
		
		\item $\Gbf^\der(\Rbb)$ acts on $V_i=\Cbb^n$ through $\SU(p_i,q_i)$, preserving an Hermitian form of signature $(p_i,q_i)$. The complexification $V_i\otimes_\Rbb\Cbb$ gives rise to a locally constant sheaf in fiber $\Cbb^{2n}$, which underlies a PVHS with a decomposition of the form $$\Ecal_i=\Ecal'^+_i\oplus\Ecal'^-_i\oplus\Ecal_i''^+\oplus\Ecal_i''^-$$ such that \begin{itemize}
			\item $\Ecal'^+_i\oplus\Ecal''^+_i$ is the complexification of the homogeneous vector bundle on $X_i^+$ (the Hermitian symmetric domain of $\SU(p_i,q_i)$) given by the action of the maximal compact subgroup of $\SU(p_i,q_i)$ on $\Cbb^{p_i}$ (through the compact unitary group $\Ubf(p_i)$), and the complex conjugation on $\Cbb^{p_i}$ induces a permutation of $\Ecal'^+_i$ and $\Ecal''^+_i$;
			
			\item similarly, $\Ecal'^-_i\oplus\Ecal''^-_i$ is associated to the negative part $\Cbb^{q_i}$ on which the maximal compact subgroup of $\SU(p_i,q_i)$ acts through the compact unitary group  $\Ubf(q_i)$.
		\end{itemize}
		Clearly the 1st Chern class of $\Ecal_i$, namely the sum of Chern classes of the summands described above, is zero as it is associated to a locally constant sheaf. The symmetry of complex polarization and the signature of Hermitian form imply that: \begin{itemize}
			\item $c_1(\Ecal'^\pm_i)+c_1(\Ecal''^\pm_i)=0$
			
			\item $\Ecal_i^{1,0}=\Ecal_i'^+\oplus\Ecal_i''^-$ is the direct sum of two subbundles of equal Chern classes and rank $p_i,q_i$ respectively.
		\end{itemize}
		
	\end{itemize}
	
	Assume that a curve $C$ (not necessarily projective) is contained in $Z\cap\Tcal_g^\circ$. The construction in \cite{clz-compositio} produces a semi-stable surface fibration $\fbar:\Sbar\ra\Bbar$, where $\Bbar$ is the compactification of the lifting $B\subset\Mcal_g$ from $C\subset\Tcal_g^\circ$. The Hodge bundle $\Ecal_{\Bbar}^{1,0}:=\fbar_*\omega_{\Sbar/\Bbar}$ admits a decomposition obtained by pulling back and extending the direct sum $\Ecal^{1,0}$ over $Z$: $$\Ecal^{1,0}_\Bbar=\Ecal_{0,\Bbar}^{1,0}\bigoplus(\Ecal'^+_{i,\Bbar}\oplus\Ecal''^-_{i,\Bbar}\oplus\cdots\oplus\Ecal'^+_{r,\Bbar}\oplus\Ecal''^-_{r,\Bbar})^{\oplus m}$$ where $\Ecal'^+_{i,\Bbar}$ and $\Ecal''^-_{i,\Bbar}$ are of equal degree $d_i>0$ and rank $p_i,q_i$ respectively, and $\Ecal_{0,\Bbar}^{1,0}$ is trivial of rank $n_0$.
	
	We thus have $\deg\Ecal_{\Bbar}^{1,0}=2m(d_1+\cdots+d_r)$ and the maximal slope in the Harder-Narasimhan filtration of $\Ecal_\Bbar^{1,0}$ is at least $\mu:=\max\{\frac{d_i}{q_i}:1\leq i\leq r,q_i\neq 0\}$. For $\fbar:\Sbar\ra\Bbar$ Xiao's slope inequality implies (cf. \cite{xiao slope} ) $$12\deg\Ecal^{1,0}_{\Bbar}\geq(2g-2)\mu.$$ We have assumed that $q_{i}<\frac{n}{12}$ for any $1\leq i\leq r$. Thus $\mu>\frac{12d}{n}$ for $d=\max\limits_i{d_i}$.
	Henceforth
	$$12\cdot 2mrd\geq 12\deg\Ecal^{1,0}_\Bbar\geq(2g-2)\mu>2(n_0+rmn)\cdot\frac{12d}{n},$$
	and one obtains $12rmn> 12(n_0+rmn)$ which is absurd.	\end{proof}

It should be pointed out that one does not expect (CO) to be true for arbitrary weakly special subvarieties, or even Shimura subvarieties.
%The cyclic covers of $\Pbb^1$ already produce positive dimensional Shimura subvarieties contained generically
%in the Torelli locus; namely, using the terminology introduced in \cite{moonen cyclic},
%one sees that the Shimura subvariety $S(\mu_m)\subset \cala_g$
%determined by the action of $\mathbb Z[\mu_n]$ has the property that
%the intersection $S(\mu_m)\cap \calt_g^{\circ}$ contains $Z(m,N, a)$,
%which is of positive dimension once $N>3$.
Prof. Martin M\"oller has communicated to us a counter-example in each dimension: the Hilbert modular variety $M$ in $\Acal_g$ associated to a totally real number field of degree $g$ always contains a Teichm\"uller curve in $\Tcal_g^\circ$, and of course infinitely many points are found in the intersection; the proof relies on \cite[\S\,6]{bouw moeller}.

Nevertheless, we would like to propose further the following conjecture for Shimura varieties of unitary type, motivated by the theorem just proved:
\begin{conjecture}
	Let $M\subset \Acal_g$ be a Shimura subvariety of the type in \autoref{unitary refinement}. Then the intersection $\Tcal_g^\circ\cap M$ is empty.
\end{conjecture}
\section{Evidence in dimension zero}\label{sec-ekedahl-serre}
In this section, we explore evidence
supporting \autoref{question-clz} for zero-dimensional weakly special subvarieties of certain types.
\autoref{thm-ekedahl-serre} (resp. Theorems \ref{thm-cm-elliptic} and \ref{thm-ekedahl-serre-2}) will be proven in
\autoref{subsec-self} (resp. \autoref{sec-variants}).

\subsection{The self-product case}\label{subsec-self}
This subsection is devoted to  \autoref{thm-ekedahl-serre}.
We first briefly recall the Sato-Tate equidistribution for a fixed elliptic curve $E$ over a number field.
Then applying the finiteness of elliptic curves with bounded Faltings height over number fields of bounded degree,
we reduce \autoref{thm-ekedahl-serre} to \autoref{thm-main}, where only one fixed elliptic curve is involved.
Finally, mimic Kukulies' idea in the geometric case,
we complete the proof of \autoref{thm-main} and hence also \autoref{thm-ekedahl-serre}
based on the Arakelov theory on arithmetic surfaces.

\vspace{2mm}
The Sato-Tate equidistribution is now available for all CM elliptic curves and ``many'' non-CM ones,
formulated in terms of the distrubition of angles for normalized Frobenius traces:
if $E$ is an elliptic curve over some number field $F$ with a  model $\Ecal$ over the integer ring $O_F$,
then at a prime $\pfrak$ in $O_F$ of good reduction for $\Ecal$ we have the Frobenius trace on the geometric special fiber:
$$a_\pfrak=\mathrm{tr(Frob_\pfrak)}=2\sqrt{q_\pfrak}\cos\theta_\pfrak$$
with $\mathrm{Frob}_\pfrak$ the geometric Frobenius acting on $H^1(\Ecal_{\kbar_\pfrak},\Qbb_\ell)$,
and the distribution of the angles $\theta_\pfrak$ in $[0,\pi]$
along the growth of the norms of primes $q_\pfrak=\#k_\pfrak$ is the subject of Sato-Tate conjecture:

\begin{theorem}[non-CM case]\label{equid-ncm}
	Let $F$ be a totally real number field, and let $E$ be an elliptic curve over $F$ of non-CM type admitting bad reduction of multiplicative type at some finite place of $F$. With the notations  used above,  the angles $\theta_\pfrak$ are equidistributed on the interval $[0,\pi]$ with respect to the Sato-Tate measure $\mu=\frac{2}{\pi}\sin^2\theta d\theta$ when $\pfrak$ runs through primes of good reduction for $E$ over $F$. In particular $$\lim_{N\ra\infty}\frac{\#\{\pfrak:q_\pfrak<N,\theta_\pfrak\in[\alpha,\beta]\}}{\#\{\pfrak:q_\pfrak<N\}}=\int_{[\alpha,\beta]}\mu(d\theta)$$ for any interval $[\alpha,\beta]\subset[0,\pi]$.%,	where $q_\pfrak$ is the cardinality of the residue field at the prime $\pfrak$.
	
\end{theorem}This is the non-CM case of the Sato-Tate conjecture proved by Clozel, Harris, Shepherd-Barron and Taylor by establishing the analytic properties of the L-functions of interest through potential automorphy lifting, see \cite{car-08} for a brief introduction.

On the other hand, the CM case is well known after the classical work of Deuring on Hecke L-series (cf. \cite{deuring}).
\begin{theorem}[CM case]\label{equid-cm}
	Let $E$ be an elliptic curve over a number field $F$ admitting CM by the integer ring $O_K$ of some imaginary quadratic number field $K$. Assume that $F$ contains $K$. Then the Sato-Tate equidistribution holds for the Frobenius eigenvalues of $H^1(E_{\kbar_\pfrak},\Qbb_\ell)$ with respect to the Sato-Tate measure on $[0,\pi]$ induced by the normalized Haar measure on the circle $S^1$.
\end{theorem} Note that a similar result holds when $K$ is not contained in $F$, but the Sato-Tate measure is slightly different, which we skip for simplicity.

We now state the result with only one fixed elliptic curve involved,
based on which we first prove the main result \autoref{thm-ekedahl-serre}.

\begin{theorem}\label{thm-main}
	Let $c$ be a fixed constant and $E$ be an elliptic curve over a number field $F$ admitting a semi-stable model $\mathcal{E}$ over $O_F$ and satisfying the Sato-Tate equidistribution.  Then there exists a constant $G=G(E,F,c)>0$ such that $g<G$ holds 	whenever there exists a smooth projective curve $C$ of genus $g$ over $F$  with a semi-stable model $\Ccal$ over $O_F$ and with an isogeny over $F$ from $J=\Jac(C)$ to the $g$-fold self-product $A=E^g$ of degree at most $c$.
\end{theorem}

\begin{proposition}\label{prop-reduction}
	\autoref{thm-ekedahl-serre} is reduced to \autoref{thm-main}.
\end{proposition}
\begin{proof}
	For fixed positive constants $d$ and $H$,
	it is clear from \cite[Corollary\,2.5]{silverman} that there exist only finitely many pairs $(F,E)$
	consisting of a number field $F$ and an $F$-isomorphism class of elliptic curve $E$	satisfying
	$$[F:\Qbb]\leq d,\qquad\text{and}\qquad h_{\Fal}(E)\leq H.$$
	For such a fixed pair $(F,E)$, 	there exists a number field $F'$, determined by $E$ and $F$,
	such that the extension of $E$ over $F'$ admits a semi-stable model $\Ecal \to F'$.
	Moreover, whenever there exists a smooth projective curve $C/F$ of genus $g$
	with an isogeny over $F$ from $J=\Jac(C)$ to the $A=E^g$ of degree at most $c$,
	so it is with the curve $C/F'$ obtained by scalar extension.
	In particular, the local monodromy for $C/F'$ is trivial (resp. unipotent)
	if and only if so is for $E/F'$,
	using the stable reduction criterion (cf. \cite[Prop\,5.11~\&~Thm\,6.3]{casalaina-13}) of Grothendieck, Deligne and Mumford.
	Thus \autoref{thm-main} produces a bound on $g$ determined by $c$ and the pair $(F,E)$.
	The proof is completed by the finiteness of the pairs $(F,E)$.
\end{proof}

To complete the proof of \autoref{thm-ekedahl-serre}, it suffices to prove \autoref{thm-main},
which can be considered as an analogue over number fields of the result from \cite{kuk-10}.
For readers' sake, we briefly review Kukulies' idea.
One starts with a function field $F$ in characteristic $p$,
a smooth projective curve $S_F$ over $F$ whose Jacobian $J_F=\Jac(S_F)$
is isogenous to the $g$-fold product of some elliptic curve $E_F$ over $F$.
Assume that these structures admit non-isotrivial semi-stable models:
$F$ is the function field of some smooth projective geometrically connected curve $C$ over $\Fbb_q$,
$S_F$ is the generic fiber of some semi-stable surface fibration $f:S\ra C$,
$E_F$ admits a semi-stable model $E\ra C$, and the isogeny $J_F\ra E_F^g$ also extends to
$J=\Pic^\circ(S/C)\ra E^g$ over $C$. The estimation of Kukulies in this case goes as follows:
\begin{itemize}
	\item the height inequality of Szpiro for $S\ra C$ bounds the number $\delta$ of the singular points
	from all the (singular) fibers of
	$S\ra C$ as $$\delta\leq 12 \deg(\omega_{S/C})$$
	with $\omega_{S/C}$ the dualizing sheaf for $S\ra C$,
	and $\deg(\omega_{S/C})=h(J/C)=gh(E/C)$ with $h(\bullet/C)$ the height of a group $C$-scheme,
	i.e., degree of the invariant differential sheaf along the neutral section;
	
	\item the Sato-Tate equidistribution for $E\ra C$ shows for $g>q^n+1$,
	the number of singular fibers of $S\ra C$ over points from $C(\Fbb_{q^n})$ is at least $\frac{1}{4}q^n$,
	and the number of singular points in  such a fiber of $S\ra C$ is at least $\lfloor \frac{g}{2q^n}\rfloor$;
	 some counting rearrangements lead to the estimate $\delta\geq c(E/C)g\frac{\log g}{\log\log g}$,
	which bounds $g$ in constants determined by $E\ra C$.
\end{itemize}

In the same spirit, we are going to prove \autoref{thm-main} using Arakelov geometry on arithmetic surfaces.
%\begin{theorem}
%	Let $c,d,H$ be a fixed constants.
%%	and $E$ be an elliptic curve over a number field $F$ admitting a semi-stable model $\mathcal{E}$ over $O_F$ and
%% satisfying the Sato-Tate equidistribution as in Theorems \ref{equid-ncm} and \ref{equid-cm}.
%	Then there exists a constant $G=G(c,d,H)>0$ such that $g<G$ holds if
%	there exists an elliptic curve $E$ over a number field $F$ with $[F:\Qbb]\leq d$
%	satisfying the Sato-Tate equidistribution as in Theorems \ref{equid-ncm} and \ref{equid-cm},
%	and a smooth projective curve $C$ of genus $g$ over a number field $F$
%	with an isogeny over $F$ from $J=\Jac(C)$ to the $g$-fold self-product $A=E^g$ of degree at most $c$.
%\end{theorem}
The counterpart over number fields of the height inequality of Szpiro is the following:

\begin{theorem}\label{number field szpiro} Let $C$ be a smooth projective curve of genus $g\geq 1$ over $\Qac$, with $\Ccal$ a semi-stable minimal model of $X$ over $B=\Spec O_F$ for the integer ring of some number field $F\subset\Qac$. Write $J=\Jac(C/F)$ and $\Jcal=\Pic^\circ(\Ccal/B)$ for the integral model of $J$ over $B$ which is a semi-abelian $B$-scheme. Then holds the inequality $$\Delta(\Ccal/B)< 12 h_\Fal(C)+6g\log2\pi^2,$$ where:
	\begin{itemize}
		\item $\Delta(\Ccal/B)=\frac{1}{[F:\Qbb]}\sum_{\pfrak}\#\Sing(\Ccal(\kbar_\pfrak))\log q_\pfrak$ is the weighted sum of singularities for  $h:\,\Ccal\ra B$,
		where $\Sing(\Ccal(\kbar_\pfrak))$ is the set of singular points in $\Ccal(\kbar_\pfrak)$ over a prime $\pfrak$ of bad reduction, and $q_\pfrak$ is the cardinality of the residue field at the prime $\pfrak$;
		
		\item $h_\Fal(C)=\frac{1}{[F:\mathbb{Q}]}\cdot \widehat{\deg}h_*\omega_{\Ccal/B}$ is the (stable)
		Faltings height of $C$, also equal to the Faltings height $h_\Fal(J)$ of the Jacobian $J=\Jac(C)$.
		%(which can be computed explicitly using their semi-stable models).
	\end{itemize}
	
\end{theorem}

The proof of this theorem is immediate after combining results in Arakelov geometry by Faltings with some recent improvement by Wilms, kindly explained to us by Prof. Ariyan Javanpeykar:

\begin{theorem}[Faltings, cf. \cite{fal-86}]
	For the arithmetic surface $\Ccal\ra B$ in \autoref{number field szpiro} holds the Noether formula $$12 h_\Fal(C)=\Delta(\Ccal/B)+e(C)+\delta_\Fal(C)-4g\log2\pi,$$
	where $\delta_\Fal(C)=\frac{1}{[F:\mathbb{Q}]}\cdot \sum\limits_{\sigma:\,F \hookrightarrow \Cbb} \delta(C_\sigma)$
	is the (stable) delta invariant of $C$ (also viewed as the archimedean discriminant),
	and $e(C)=\frac{1}{[F:\mathbb{Q}]}\cdot \omega_{\Ccal/B}^2$.
	Moreover, $e(C)\geq 0$.	
%	is the archimedean discriminant computed using the Riemann surface structures on
%	$C(\Cbb)$ obtained from different embeddings $F\mono \Cbb$. Moreover $\omega^2(C)\geq 0$.
\end{theorem}

%Very recently R. Wilms ( \cite{wil-16}, Theorem 1.1) finded a relation bewteen Faltings $\delta$-invariant and Zhang-Kawazumi invariant [?], [??]. Consequently he obtained

{ The improvement from Wilms \cite{wil-16} provides a link between the $\delta$-invariant of Faltings and the $\varphi$-invariant of Kawazumi-Zhang, cf. \cite{kawazumi} and \cite{zhang}.
	For simplicity we only mention the following less precise consequence, which is sufficient for the analogue of Szpiro inequality:
}
\begin{theorem}[Wilms, cf.\cite{wil-16}] For $\Ccal\ra B$ the arithmetic surface as above holds the inequality $\delta_\Fal(C)> -2g\log 2\pi^4$.
	
\end{theorem}

The proof of \autoref{thm-main} is thus reduced to an estimation of singular points, along the idea in \cite{kuk-10}:

\begin{proof}[Proof of \autoref{thm-main}]
	%For $\pfrak$ a prime of good reduction for $E$ over $F$, let $q_\pfrak$ be the cardinality of the residue field at the prime $\pfrak$, and $a_\pfrak$ be the trace of the $q_\pfrak$-Frobenius on $H^1(E_{\kbar_\pfrak},\Qbb_\ell)$.
	Keep the notations $a_\pfrak$, $q_\pfrak$ etc. as before. The assumption on the equidistribution of Sato-Tate type  implies that the following subset of primes of good reduction for $E$ over $F$	$$P=\{\pfrak:a_{\pfrak}>0\}$$ is of density $\rho>0$ in the set of primes of $F$.
	Note that the Frobenius trace is an integer.
	It follows that $a_{\pfrak}\geq 1$ for any $\pfrak\in P$.

	On the other hand,  a prime $\pfrak$ of good reduction for $\Ccal$ is also of good reduction for $\Ecal$ via the isogeny between $J\otimes_FF_\pfrak$ and $E^g\otimes_FF_\pfrak$ using Serre-Tate's criterion, cf. \cite{serre tate}. The  isomorphism	$$H^1(\Ccal_{\kbar_\pfrak},\Qbb_\ell)\isom H^1(\Jcal_{\kbar_\pfrak},\Qbb_\ell)\isom H^1(\Ecal_{\kbar_\pfrak},\Qbb_\ell)^{\oplus g}$$	leads to the point counting $$\#\Ccal(k_\pfrak)=1+q_\pfrak-ga_\pfrak,$$	which together with the inequality $\#\Ccal(k_\pfrak)\geq 0$ implies that
	$$1+q_\pfrak-g\geq 0,\qquad \forall~\pfrak\in P.$$	In particular, $q_\pfrak+1\geq g$ for such primes of good reduction.%We simply keep $q_\pfrak\geq g$ in what follows.
		
	Put	$$P(g):=\{\pfrak \in P~|~q_\pfrak+1 < g\}~\subseteq P.$$	Since the set $P$ defined above is of density $\rho>0$, the subset $P(g)$ is of cardinality at least $\half\rho\cdot \frac{g}{\log g}$ when $g$ is large enough. Moreover, for $\pfrak\in P(g)$, although $J$, or equivalently $E$, is of good reduction, the fiber $\Ccal_{k_\pfrak}$ must be singular due to the counting inequality $q_\pfrak+1\geq g$ established above for primes of good reduction of $\Ccal\ra B$.
	
	We claim that the number of singular points in such a singular fiber is at least $\frac{g}{2q_\pfrak}$. The argument is the same as in \cite{kuk-10}: when $q_\pfrak+1<g$, the Jacobian is either a torus, or isogenous to the $g$-fold product of a single elliptic curve. In the toric case the curve $\Ccal_{\kbar_\pfrak}$ has at least $g$ singular points; in the compact case the curve is a chain of smooth curves each of genus  at most $q_\pfrak+1$, and at least $\frac{g}{2q_\pfrak}$ singular points are found in the fiber.  
	
	Summing over these primes in $P(g)$ gives the following inequality:
	$$\Delta(\Ccal/B)\geq\frac{1}{d}\sum_{\pfrak\in P(g)}\frac{g}{2q_\pfrak}\log q_\pfrak\geq \frac{1}{2d}\cdot \rho_1 g\log g,$$
	for some $\rho_1\in(0,\rho)$ with $d=[F:\Qbb]$, using some standard estimation in analytic number theory
	(cf. Lemma \ref{lemma analytic number theory} below).

	It is also known from \cite{fal-83} (or \cite{ray-85}) that for given $\Ecal$ and $F$,	the Faltings height of an abelian variety $A'$ over $F$ admitting an $F$-isogeny to $E^g$	of degree $f$ only differs from $gh_\Fal(\Ecal/O_F)$ by a quantity bounded by $\log \deg f$.	Since we have required the degree of isogeny between $J$ and $E^g$ be  bounded by $c$, we deduce from Theorem \ref{number field szpiro} that	$$\Delta(\Ccal/B)\leq 12h_\Fal(C)+gO(1)=g\big(12h_\Fal(E)+O(1)\big)+O(1),$$ 	where $O(1)$ stands for some constant determined by $\Ecal,F$ and $c$, independent of $g$.	Thus	$$\frac{1}{2d}\rho_1g\log g\leq g\big(h_{\Fal}(E)+O(1)\big)+O(1).$$
	Eliminating the linear factor $g$ one obtains an upper bound of $g$ as required.
%	in terms of $h_\Fal(E)$, $d=[F:\mathbb{Q}]$ and $c$. 
\end{proof}

\begin{remark}
(1)	Note that the bounds in the theorems are not explicit. On the other hand, if we consider the case of function field over the field $\Cbb$, namely a semi-stable surface fibration $\Sbar\ra\Bbar$ with $\Bbar$ some algebraic curve over $\Cbb$ such that the Picard variety $\Pic(\Sbar/\Bbar)$ is generically a Jacobian over the function field $F$ of $\Bbar$ isogenous to a $g$-fold self-product of some elliptic curves over $F$, then \cite{lz14} affirms an upper bound $g\leq 11$. However, the proof loc. cit. relies heavily on transcendental results such as the logarithmic Miyaoka-Yau inequality, which has no arithmetic analogy at the moment.

(2) The next section presents a similar result of such finiteness for products of CM abelian varieties, making use of recent progress on Sato-Tate equidistribution and Faltings heights.
\end{remark}

In the proof of \autoref{thm-main} above we have made use of some standard estimation from analytic number theory,
based on the useful fact that for $\{a_n\}$ a sequence of numbers and $b(x)$ a function of $C^1$-class on $\Rbb_{\geq 0}$, holds the following identity (see for example \cite{mv-07}):
\begin{equation}\label{eqn-2-1}
\sum_{n\leq x}a_nb(n)=A(x)b(x)-\int_1^xA(t)b'(t)dt,
\end{equation}
with $A(t)=\sum\limits_{n\leq t}a_n$.

\begin{lemma}\label{lemma analytic number theory}
	(1) Let $P$ be the set of prime numbers in $\Nbb$, and $Q$ a subset of natural density $c$, i.e. $$\lim_{x\ra\infty}\frac{\#Q\cap P(x)}{\#P(x)}=c,$$ where $P(x)=\{p\in P:p\leq x\}$. Then asymptotically holds
	$$\sum_{p\in Q\cap P(x)}\frac{\log p}{p}\sim c\cdot (\log x-\log\log x), \qquad x\ra\infty.$$
	
	(2) Let $F$ be a number field, $P_F$ its set of prime ideals and $Q_F\subset P_F$ a subset of density $c>0$, namely $$\lim_{x\ra\infty}\frac{\# Q_F\cap P_F(x)}{\# P_F(x)}=c,$$	where similar as above we have $P_F(x)=\{\wp\in P_F:q_\wp\leq x\}$	with $q_\pfrak$ being the cardinality of the residue field at the prime $\pfrak$. Then asymptotically holds the estimation	$$\sum_{\wp\in Q_F\cap P_F(x)}\frac{\log q_\wp}{q_\wp}\geq c_1 \log x,\qquad x\ra \infty,$$	for some contant $c_1$ depending only on $Q_F$ and $F$.
\end{lemma}

\begin{proof}
	
	(1) Consider the sequence $\{a_n\}$ given by $a_p=1$ for $p\in Q$ and zero otherwise. Then $$A(x):=\sum_{n\leq x}a_n \sim c\frac{x}{\log x}.$$	The summation formula \eqref{eqn-2-1} above gives
	$$\begin{aligned}
	\sum_{p\in Q\cap P(x)}\frac{\log p}{p} &\,=\sum_{n\leq x}a_n\frac{\log n}{n}\\
	&\,\sim c\frac{x}{\log x}\cdot\frac{\log x}{x}-\int_2^xc\frac{t}{\log t}d(\frac{\log t}{t})=c\log x-c\log\log x+O(1).
	\end{aligned}$$

	(2) Write $p_F(n)$ for the number of ways representing $n\in\Nbb$ as the norm of a prime ideal from $P_F$. Then Landau's prime number theorem $\sum_{\wp\in P_F(x)}1\sim \frac{x}{\log x}$ for $F$ is the same as $$\sum_{n\leq x}p_F(n)\sim \frac{x}{\log x}.$$ Similarly one writes $q_F(n)$ for the number of ways representing $n$ as the norm of a prime ideal from $Q_F$, and the assumption on natural density is $$\lim_{x\ra\infty}\frac{\sum_{n\leq x}q_F(n)}{\sum_{n\leq x}p_F(n)}=c.$$ Applying the summation formula mentioned above we obtain $$\sum_{n\leq x}q_F(n)\frac{\log n}{n}\sim c(\log x-\log\log x),$$ and one may simply take $c_1=\frac{c}{2}<c$.
\end{proof}

\subsection{variant for products of CM abelian varieties}	\label{sec-variants}

The Sato-Tate equidistribution has gone through various forms of generalization in recent years, cf. \cite{fite equidistribution} for related discussion on notions such as Sato-Tate groups, etc. For general CM abelian varieties over number fields the equidistribution is established, as a consequence of the potential automorphy theory, cf. \cite{johansson}:

\begin{theorem}\label{johansson}
	Let $A$ be an abelian variety over a number field $F$, which is isogeneous to a product of simple CM abelian varieties after a finite base change of $F$. Then the Sato-Tate equidistribution holds for $A$.
\end{theorem}

In \cite{johansson} this is worked out in the more general setting of semi-simple potentially abelian geometric Galois representations of pure weights. The precise formulation of equidistribution involves the notion of Sato-Tate group: in the case of interest, e.g. a CM abelian variety $A$, the Sato-Tate group is a compact torus (products of $S^1$), identified as a maximal compact subgroup of the complex locus of the Hodge group of $A$, and the limit measure is simply the Haar measure on the Sato-Tate group.

Let $F$ be a fixed number field and let $A_1,\cdots,A_r$ be  CM abelian varieties over $F$, of good reduction over a common Zariski open set of prime ideals $P$ of $O_F$. If $A$ is an abelian variety over $F$ isogeneous to a product of the form $A_1^{m_1}\times\cdots\times A_r^{m_r}$, then the theorem above affirms the existence of a subset $Q$ in $P$ of strictly positive density such that tuples of normalized Frobenius traces $$(\lambda_\pfrak(1),\lambda_\pfrak(1),\cdots,\lambda_\pfrak(r),\lambda_\pfrak(r))$$ (with $\lambda_\pfrak(i)$ repeated $m_i$ times) fall in any given product of intervals $[a,b]^{m}$ inside $[-1,1]^m$ ($m=\sum_im_i$, $a<b$) when $\pfrak$ runs through $Q$; here $\lambda_\pfrak(i)=\cos\theta_\pfrak(i)$ is the Frobenius trace of $A_i$ at $\pfrak$ divided by $2\sqrt{q_\pfrak}$.  For a natural generalization of the main result in the self-product case it remains to input bounds on Faltings height, for which the following estimations are available: 

\begin{theorem}\label{thm-bound-falting}
	Let $L$ be a CM field and $A$ an $r$-dimensional simple abelian variety with CM by $L$, defined over a number field $F$. Then:
	
	(1) the Faltings height is bounded from above in the form $$h_\Fal(A)\leq |\Disc(L)|^{o_r(1)}$$ where $o_r(1)$ is an $\Rbb_{\geq0}$-valued function tending to zero as $r$ grows, independent of $A$;
	
	(2) there exists constant $c_r>0,\delta_r>0$, independent of $A$ such that $$[F:\Qbb]\geq c_r|\Disc L|^{\delta_r}$$
\end{theorem}

Here (1) is a consequence of the average Colmez conjecture proved by \cite{andreatta colmez} and \cite{yuan zhang colmez}, and we have only taken the convenient form following \cite{tsimerman siegel}; (2) is also found in \cite{tsimerman siegel} which affirms that the Galois orbits of CM points are large compared to the discriminants involved.  Combining the two results one finds that once the field of definition $F$ is of fixed degree, there are only finitely CM abelian varieties defined over $F$ of bounded dimension admitting semi-stable models over $O_F$, due to the bound of Faltings heights. We are thus ready to prove the following theorems:

\begin{theorem}\label{thm-4.14}
	Fix integers $d>0,R>0,c>0$, and write $S_g(d,R,c)$ for the set of points $[J]$ in $\Tcal_g^\circ(\Qac)$ subject to the following conditions: \begin{itemize}
		\item $J$ is the Jacobian of a smooth projective curve $C$ defined over a number field $F$ of degree at most $d$, of genus $g$, admitting a semi-stable model $\Ccal$ over $O_F$;
		
		\item there exists an isogeny of degree at most $c$ from $J$ to a product $A_1\times\cdots\times A_m$ where \begin{itemize}
			\item the $A_i$'s are CM abelian varieties of dimension at most $R$ defined over $F$, with semi-stable models over $O_F$;
			
			\item the $A_i$'s satisfy the Sato-Tate equidistribution.
		\end{itemize}
	\end{itemize}	
	%that admits an isogeny of degree at most $c$ to a product $A_1\times\cdots\times A_m$ where $A_i$ are CM abelian varieties of dimension at most $R$ defined over a common number field $F$ of degree at most $D$ such that the Sato-Tate equidistribution holds for all the $A_i$'s over $F$ and that the $A_i$'s admit . 	
	Then there exists a constant $G=G(d,R,c)$
	such that $S_g(d,R,c)$ is finite if $g<G$, and that $S_g(d,R,c)$ is empty if $g\geq G$.
\end{theorem}

\begin{proof}
	According to \autoref{thm-bound-falting},
	the Faltings heights of CM abelian varieties of dimension at most $R$ defined over a number field of degree at most $d$
	are bounded from above.
	Write $H$ for the maximum of Faltings heights of CM abelian varieties involved as above.
	
	%	Write $H$ for the maximum of Faltings heights of CM abelian varieties involved as above,
	%i.e. defined over $F$ of dimension at most $R$, which exists as the set of isomorphism classes of
	%such abelian varieties is actually finite due to the bound in degree,
	%which in turn bounds the discriminants and the Faltings heights.
	
	Let $J$ be the Jacobian of some smooth projective curve $C$ defined over $F$ with a semi-stable $O_F$-model $\Ccal$.
	Similar to \autoref{subsec-self}, we have the following inequality
	$$\Delta(\Ccal/O_F)\leq h_\Fal(J)+gO(1)+d_c\leq g(H+O(1))+d_c,$$
	where $d_c$ is a constant that only depends on $c$ the bound of isogeny degree.
	
%	Each $A_i$ being of dimension at most $R$, one finds at least $\lfloor \frac{g}{R}\rfloor$
%	factors in the product $A_1\times\cdots\times A_m$, hence at a prime of bad (semi-stable) reduction
%	$\pfrak$ for $\Ccal$ over $O_F$, the number of singular points in $\Ccal(\kbar_\pfrak)$
%	is at least $g\cdot\frac{1}{2Rq_\pfrak}$.
	
	On the other hand, the Sato-Tate equidistribution assures the existence of a set of prime ideals $Q$ in $O_F$ of positive density,
	such that for any $\pfrak\in Q$, the abelian variety $A_i$ is of good reduction whose Frobenius trace $a_{\pfrak,i}$
	is a strictly positive integer, $i=1,\cdots,m$. Similar to the proof of \autoref{thm-main},
	using the non-negativity of the $k_\pfrak$-rational points on $C_{k_\pfrak}$,
	one deduces that $g\leq R(1+q_\pfrak)$ if $\pfrak$ were also a prime of good reduction for $\Ccal$ over $O_F$.
	Hence for $\pfrak\in Q$ with $q_\pfrak<\frac{g}{R}-1$ does contribute non-trivially to $\Delta(\Ccal/O_F)$ as we do in the proof of \autoref{thm-main},
	and one is led to
	$$\frac1{2Rd}\cdot \rho_1g\log g\leq \frac{1}{d}\sum_{\pfrak\in Q,~q_\pfrak<\frac{g}{R}-1}g\frac{\log q_\pfrak}{2Rq_\pfrak}
	\leq\Delta(\Ccal/O_F)\leq g(H+O(1))+d_c,$$
	for some constant $\rho_1>0$.
	Hence one gets an upper bound on $g$ in terms of $d$, $R$ and $c$.
	
	Finally, for fixed genus $g\geq 1$, one knows that
	$$h_{\Fal}(C)=h_\Fal\big(\Jac(C)\big) \leq g(H+O(1)),$$
	for any smooth projective curve $C\in S_g(d,R,c)$.
	Hence $S_g(d,R,c)$ is finite by Northcott's theorem.
\end{proof}

In the case of CM elliptic curves, the Sato-Tate equidistribution is known \autoref{equid-cm}.
\autoref{thm-cm-elliptic} is now a direct consequence of \autoref{thm-4.14}.
We end this subsection with the proof of \autoref{thm-ekedahl-serre-2}.
\begin{proof}[Proof of \autoref{thm-ekedahl-serre-2}].
	Let $C$ any smooth projective curve of genus $g>0$ defined over a number field $F$ with $[F:\mathbb{Q}]\leq d$,
	and the Jacobian $\Jac(C)$ is isogenous to $E^g$ for some CM elliptic curve $E$ over $F$.
	By \cite[Theorem\,2]{kani}, there exist CM elliptic curves $E_1/F,\,\cdots,\, E_g/F$, such that
	$$\Jac(C)\cong E_1 \times \cdots \times E_g.$$
	Thus by \autoref{thm-4.14}, there exists a positive constant $G=G(d)$ such that $g\leq G$.
	
	The finiteness follows from the the same reason as in \autoref{thm-4.14}.
	Indeed, according to \autoref{thm-bound-falting},
	for any CM elliptic curve $E$ over a number field of degree at most $d$,
	there exists a positive number $M$ such that
	$$h_{\Fal}(E)\leq M.$$	
	Hence for any $C\in \mathfrak{S}_d$,
	$$h_{\Fal}(C)=h_{\Fal}\big(\Jac(C)\big)=\sum_{i=1}^{g}h_{\Fal}(E_i)\leq gM\leq GM.$$
	Therefore, the set $\mathfrak{S}_d$ is finite by Northcott's theorem.
\end{proof}

{\vspace{2mm}\bf Acknowledgment.}
The authors thank Prof. Ariyan Javanpeykar heartily for very helpful discussion on Szpiro's inequality and Faltings heights, and for kindly suggesting improvements on results in \autoref{sec-variants}. They also thank Prof. Frans Oort for the discussions on totally decomposable Jacobians and Prof. Martin M\"oller for communicating the example of Teichm\"uller curves in Hilbert modular varieties.

% \bibliographystyle{amsalpha}
% %\bibliographystyle{a-alpha}
% \bibliography{ekedahl-serre}

\providecommand{\bysame}{\leavevmode\hbox to3em{\hrulefill}\thinspace}
\providecommand{\MR}{\relax\ifhmode\unskip\space\fi MR }
\providecommand{\MRhref}[2]{%
	\href{http://www.ams.org/mathscinet-getitem?mr=#1}{#2}
}
\providecommand{\href}[2]{#2}

\end{document}